\documentclass[runningheads]{llncs}
\usepackage{graphicx}
\usepackage{mathrsfs}  
\usepackage{amsmath, amssymb}
\begin{document}
\title{Tensors in modelling multi-particle interactions \thanks{The work was supported by the Russian Science Foundation, grant 19-11-00338}}
%
%
\author{Daniil A. Stefonishin \inst{1} \and
Sergey A. Matveev \inst{1,2}
\and
Dmitry A. Zheltkov \inst{2,3}}
\authorrunning{D. Stefonishin et al.}
%
\institute{Skolkovo Institute of Science and Technology, Moscow, Russia \and
Marchuk Institute of Numerical Mathematics RAS, Moscow, Russia \and Moscow Institute of Physics and Technology, Moscow, Russia
\email{s.matveev@skoltech.ru}\\
\url{www.skoltech.ru}, \url{www.inm.ras.ru}\\
\email{d.stefonishin@skotech.ru},\ \email{dmitry.zheltkov@gmail.com}}
\maketitle              
\begin{abstract}

In this work we present recent results on application of low-rank tensor decompositions to modelling of aggregation kinetics taking into account
multi-particle collisions (for three and more particles). Such kinetics can
be described by system of nonlinear differential equations with right-hand
side requiring $N^D$ operations for its straight-forward evaluation, where $N$ is
number of particles’ size classes and $D$ is number of particles colliding simultaneously.
Such a complexity can be significantly reduced by application
low rank tensor decompositions (either Tensor Train  or Canonical Polyadic)
to acceleration of evaluation of sums and convolutions from right-hand side.
Basing on this drastic reduction of complexity for evaluation of right-hand
side we further utilize standard second order Runge-Kutta time integration
scheme and demonstrate that our approach allows to obtain numerical solutions
 of studied equations with very high accuracy in modest times. We
also show preliminary results on parallel scalability of novel approach and
conclude that it can be efficiently utilized with use of supercomputers.

\end{abstract}
\section{Introduction}

\quad Aggregation of inelastically colliding particles plays important role in many technological and natural phenomena. In case of spatially homogeneous systems aggregation process can be described by famous Smoluchowski kinetic equations~\cite{Galkin_book}. These equations describe time-evolution of mean concentrations $n_k$ of particles of size $k$ per unit volume of media:
\begin{eqnarray*}
\notag
\frac{d n_k}{d t} = \frac{1}{2} \sum\limits_{i+j=k} C_{i, j} n_i n_{j} - n_k \sum\limits_{i=1}^{\infty} C_{k, i} n_k.
\end{eqnarray*}
Such a model is well-studied by lots of analytical and numerical methods but allows to take into account only pairwise particles' collisions. In this work we consider a generalization of aggregation equations for case of multi-particle interactions
\begin{eqnarray*}
    \frac{\mathrm{d} \mathbf{n}}{\mathrm{d}\hspace{0.2mm}t}
    = \sum_{d = 2}^{D}
        \mathcal{S}^{(D\hspace{0.3mm})}\!\left[\mathbf{n}\right]
    = \sum_{d = 2}^{D} \left[
        \mathcal{P}^{(D\hspace{0.3mm})}\!\left[\mathbf{n}\right]
        + \mathcal{Q}^{(D\hspace{0.3mm})}\!\left[\mathbf{n}\right]\hspace{-0.3mm}
    \right]\!,
\end{eqnarray*}
where operators~
$
    \mathcal{P}^{(d\hspace{0.3mm})}
    = \left[
        \hspace{0.3mm}p^{(d\hspace{0.3mm})}_{1}\!,
        \hspace{0.3mm}p^{(d\hspace{0.3mm})}_{2}\!,
        \hspace{1mm}\ldots\hspace{1mm}
    \right]^{T}
$
and~
$
    \mathcal{Q}^{(d\hspace{0.3mm})}
    = \left[
        q^{(d\hspace{0.3mm})}_{1}\!,
        \hspace{0.3mm}d^{(d\hspace{0.3mm})}_{2}\!,
        \hspace{1mm}\ldots\hspace{1mm}
    \right]^{T}
$
are defined as
\begin{gather}
\notag
    p^{(d\hspace{0.1mm})}_{k}\!\left[\mathbf{n}\right]
    = \frac{1}{d\hspace{0.1mm}!}\sum_{\left|\mathbf{i}_{d}\right| = k}
        C^{(d\hspace{0.1mm})}_{
            \mathbf{i}_{d}
        }\hspace{0.5mm}n_{i_{1}}n_{i_{2}}\ldots n_{i_{d}},
        \quad k \in \mathbb{N},
\\
\notag
    q^{(d\hspace{0.1mm})}_{k}\!\left[\mathbf{n}\right]
    = - \frac{n_{k}}{\left(d - 1\right)\hspace{0.1mm}!}\sum_{
        \mathbf{i}_{d - 1} \in \mathbb{N}^{d - 1}
    }
        C^{(d\hspace{0.3mm})}_{
           \mathbf{i}_{d - 1}, \hspace{0.2mm}k
        }\hspace{1mm}n_{i_{1}}n_{i_{2}}\ldots n_{i_{d - 1}},
        \quad k \in \mathbb{N},
        \\
\notag
    \mathbf{i}_{d} = \left(i_{1}, i_{2}, \ldots, i_{d}\right)\!,
    \quad 
    \left|\mathbf{i}_{d}\right| = i_{1} + i_{2} + \ldots + i_{d},
        \quad 2 \le d \le D\hspace{-0.3mm}.
\end{gather}
Those operators $p^{(d\hspace{0.3mm})}$ and $q^{(d\hspace{0.3mm})}$ correspond to description of simultaneous collisional aggregation of $d$ particles. We represent such a model informally in Fig.~\ref{Figure_aggregation}. 
\begin{figure}
\begin{center}
\includegraphics[scale=0.15]{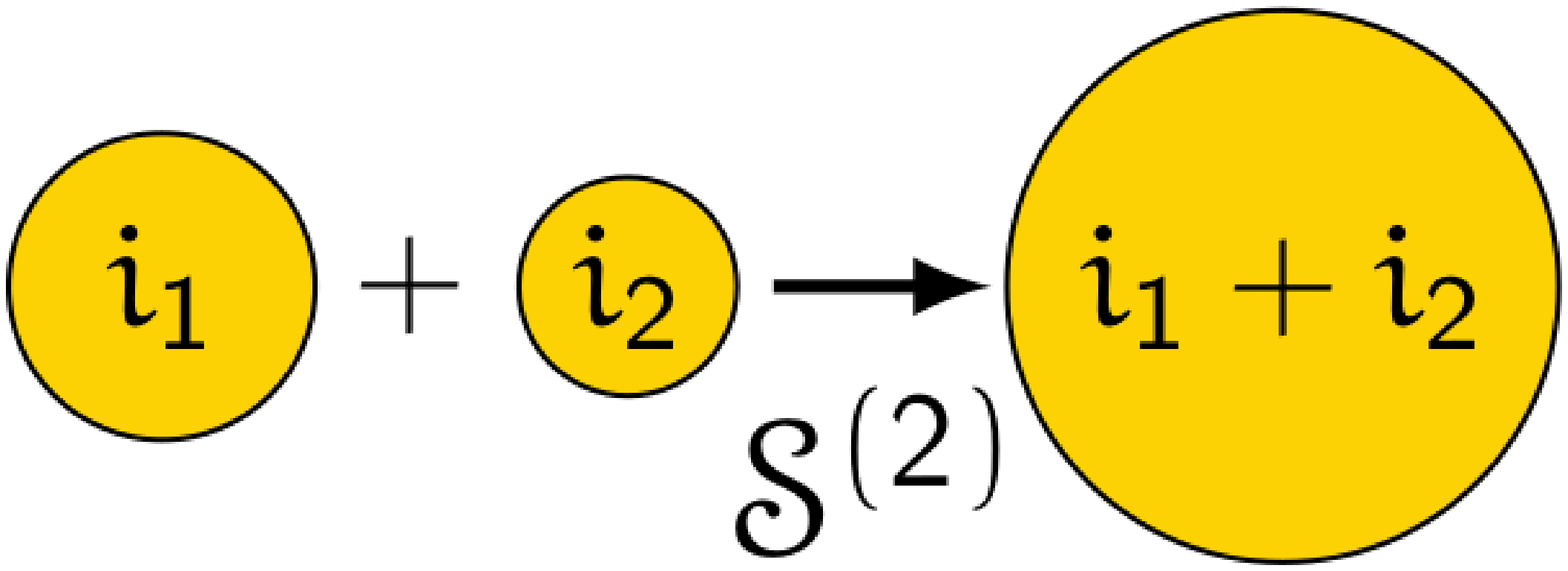}\qquad\qquad
\includegraphics[scale=0.15]{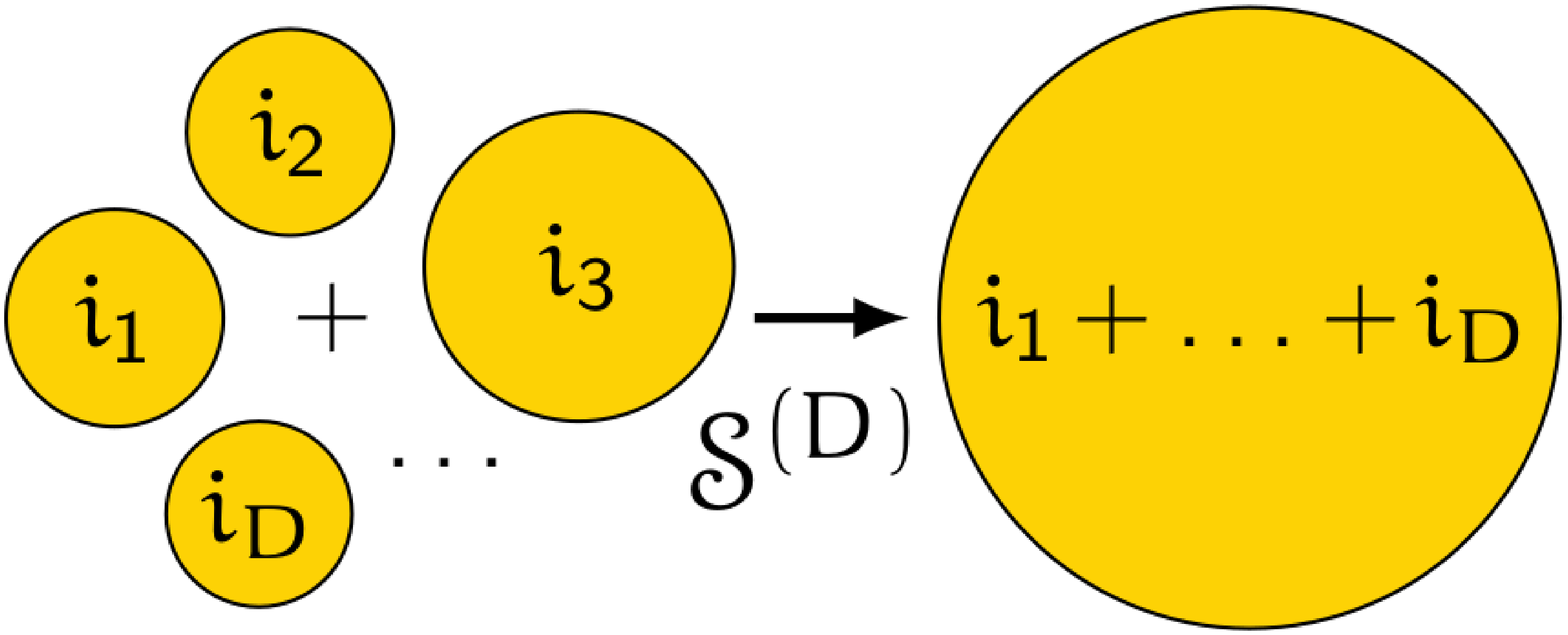}
\end{center}
\caption{On the left panel binary aggregation of particles is presented, and multi-particle collision on the right panel} \label{Figure_aggregation}
\end{figure}

In case of defined initial conditions $n_k(t = 0)$ we obtain a Cauchy problem which is known to be well-posed under assumption of bounded non-negative symmetric coefficients \cite{Galkin2013}. There also exist a very few examples of known analytical solutions for the Cauchy problem for multi-particle aggregation kinetic equations \cite{Krapivsky1991}. Unfortunately, such class of mathematical models is much less studied numerically than class of aggregation equations accounting only binary collisions. The reason lies in higher order of non-linearity and exponential growth of complexity of evaluation of the right-hand side with respect to number $D$ of simultaneously colliding particles. It makes numerical treatment of such systems extremely time-consuming. 

In our recent work we proposed a novel approach based on application of low-rank tensor train (TT) decomposition ~\cite{OselTT,TTcross} to acceleration of computations and presented its high accuracy ~\cite{Stefonishin_tensors,Stefonishin_triple}. 
This allowed us to reduce the complexity of evaluation of sums from the right-hand side from $O(N^D)$ operations to $O(N D R^2 \log N)$ operations, where $R \ll N$ is the maximal TT-rank of kinetic coefficients. In case of pre-defined low-rank Canonical Polyadic (CP) decomposition of kernel coefficients the complexity can be reduced even to $O(N D R \log N)$ but we do not know a robust way finding CP-decomposition even for 3-dimensional tensors. Nevertheless, assumption that $ R \ll N$ is crucial for efficiency of proposed approach. 

In the current work we prove that TT-ranks of a wide class of generalized Brownian kernels are low and do not depend on number $N$ of accounted kinetic equations for arbitrary dimension $D$. We also present an efficient parallel implementation of our TT-based approach and present preliminary tests of scalability our method.

\section{Estimates of TT-ranks for generalized Brownian kernels}

\quad In this section we present estimates of tensor ranks for generalized Brownian kernel~$\mathscr{C}^{(D)} = \left[C^{(D)}_{\mathbf{i}_{D}}\right]$ (see example of exact Brownian coefficients e.g. at \cite{MatveevPRL}) with the~elements of~the~following form:
 
 \begin{gather}
 \label{ch3.eq17}
        C^{(D\hspace{0.3mm})}_{\mathbf{i}_{D}}
        \equiv
        C^{(D\hspace{0.3mm})}_{\mathbf{i}_{D}}
        \!\left[
            \hspace{0.3mm}\mu_{1}, \hspace{0.3mm}\mu_{2}, \,\ldots,
            \,\mu_{D}
        \right]
        = \sum_{\sigma}
            i_{\sigma\left(1\right)}^{\,\mu_{1}}
            \cdot i_{\sigma\left(2\right)}^{\,\mu_{2}}
            \cdot\hspace{1mm}\ldots\hspace{1mm}
            \cdot i_{\sigma\left(D\hspace{0.2mm}\right)}^{\,\mu_{D}}
    \end{gather}
Here we assume the~sum to~be over all permutations~$\sigma$ of the set~$\left\{1, 2, \ldots, D\right\}$.
    
Recall the~definition of~a~TT-decomposition for a~kernel~$\mathscr{C}^{(D)}$ which is of~the~form
\begin{gather}
\label{ch3.eq4}
    C^{(D)}_{\mathbf{i}_{D}}
    = \hspace{-2mm}\sum_{
        r_{0}, \hspace{0.3mm}r_{1}, \,\ldots, \,r_{D}
    }\hspace{-2mm}
        H^{(1)}_{r_{0}, \hspace{0.3mm}i_{1}, \hspace{0.3mm}r_{1}}
        \cdot H^{(2)}_{r_{1}, \hspace{0.3mm}i_{2}, \hspace{0.3mm}r_{2}}
        \cdot\hspace{1mm} \ldots\hspace{1mm}
        \cdot H^{(D\hspace{0.3mm})}_{
            r_{D - 1}, \hspace{0.3mm}i_{D}, \hspace{0.3mm}r_{D}
        },
\\\notag
    1 \le r_{\lambda} \le R_{\lambda},
    \quad 0 \le \lambda \le D;
    \qquad R_{0} = R_{D} = 1.
\end{gather}

For~the~kernel~$\mathscr{C}^{(D)}$ of~dimension~$D$ with the~elements~\eqref{ch3.eq17} there holds 
an~estimate on the TT-ranks~$R_{\lambda}$:
\begin{gather}
\notag
    \max\!\left\{R_{\lambda}\colon 0 \leq \lambda \leq D\right\}
    = \dbinom{D}{\left\lceil D/2 \right\rceil} 
    \equiv O\!\left(2^{\hspace{0.2mm}D} / \sqrt{D}\right)\!.
\end{gather}

Such estimate can be verified by the means of the following
\begin{theorem}
\label{s3.3.th1}
    Let the parameters~$\mu_{1}, \mu_{2}, \ldots, \mu_{D}$ be fixed.
    For a~given tensor~$\mathscr{C}^{(D)}$ in~\hbox{$D$} dimensions
    of~sizes~\hbox{$N \times N \times \ldots \times N$} with~the~elements~\eqref{ch3.eq17}
    one can prove the following estimates on~its~TT-ranks~$R_{\lambda}$:
    \begin{gather}
    \notag
        R_{\lambda}
        \le \dbinom{D}{\lambda}
        \equiv
        \frac{D\hspace{0.3mm}!}{
            \lambda\hspace{0.2mm}!\cdot\left(D - \lambda\right)\hspace{0.2mm}!
        },
        \qquad 0 \le \lambda \le D\hspace{-0.3mm}.
    \end{gather}
\end{theorem}
\begin{proof}
    Let us put~$R_{\lambda} := \dbinom{D}{\lambda}$
    for~$0 \le \lambda \le D$\hspace{-0.3mm}.
    To~prove the~theorem we simply need to construct a~tensor train decomposition of~the~tensor~$\mathscr{C}^{(D)}$ with these predefined ranks~$R_{\lambda}$.

    Further we assume that for each number~$
        1 \le \lambda \le D
    $ it is chosen a~bijection~$
        r_{\lambda} \to \left(
            r_{1, \hspace{0.3mm}\lambda},
            \hspace{0.3mm}r_{2, \hspace{0.3mm}\lambda},
            \,\ldots,
            \,r_{\lambda, \hspace{0.3mm}\lambda}
        \right)
    $ between the~sets
    \begin{gather}
    \notag
        \left\{
            r_{\lambda} \in \mathbb{N}
            \colon 1 \le r_{\lambda} \le R_{\lambda}
        \right\}\!,
    \\\notag
        \mathcal{R}_{\lambda} := \left\{
            \left(
                r_{1, \lambda}, \hspace{0.3mm}r_{2, \hspace{0.3mm}\lambda},
                \,\ldots,
                \,r_{\lambda, \hspace{0.3mm}\lambda}
            \right) \in \mathbb{N}^{\lambda}
            \colon 
            1 \le r_{1, \hspace{0.3mm}\lambda}
            < r_{2, \hspace{0.3mm}\lambda}
            < \ldots < r_{\lambda, \hspace{0.3mm}\lambda} \le D
        \right\}\!.
    \end{gather}
    One can specify such mappings for~sure due~to~the~coincidence of~the~cardinalities of~the~considered sets.

    Note, that for~all~numbers~$1 \le \lambda \le D - 1$ one can
    check the~correctness of~the~identity
    \begin{gather}
    \label{ch3.eq18}
        C^{(\lambda + 1)}_{\mathbf{i}_{\lambda + 1}}\!\left[
            \hspace{0.3mm}\mu_{1}, \hspace{0.3mm}\mu_{2}, \hspace{0.3mm}\ldots,
            \hspace{0.3mm}\mu_{\lambda + 1}
        \right]
        \equiv \sum_{\xi = 1}^{\lambda + 1}
            C^{(\lambda)}_{
                \mathbf{i}_{\lambda}
            }\!\left[
                \hspace{0.3mm}\mu_{1}, \,\ldots,
                \,\mu_{\xi - 1}, \hspace{0.3mm}\mu_{\xi + 1},
                \,\ldots, \,\mu_{\lambda + 1}
            \right]
            \cdot 
            i_{\lambda + 1}^{\hspace{0.3mm}\mu_{\xi}}.
    \end{gather}
    With the~given identity we show by induction on~$1 \le \lambda \le D$,
    that it is always possible to~choose the~values~$
        H^{(\tau)}_{
            r_{\tau - 1}, \hspace{0.3mm}i_{\tau}, \hspace{0.3mm}r_{\tau}
        }
    $ in~order to satisfy the~constraints
    \begin{gather}
    \label{ch3.eq19}
        \sum_{
            r_{0}, \hspace{0.3mm}r_{1},
            \hspace{0.3mm}\ldots, \hspace{0.3mm}r_{\lambda - 1}
        }
        \hspace{-6mm}
        H^{(1)}_{r_{0}, \hspace{0.3mm}i_{1}, \hspace{0.3mm}r_{1}}
        \hspace{-1mm}\cdot 
        H^{(2)}_{r_{1}, \hspace{0.3mm}i_{2}, \hspace{0.3mm}r_{2}}
        \hspace{-1mm}\cdot 
        \ldots
        \cdot
        H^{(\lambda)}_{
            r_{\lambda - 1}, \hspace{0.3mm}i_{\lambda},
            \hspace{0.3mm}r_{\lambda}
        }\hspace{-1mm}
        = C^{(\lambda)}_{\mathbf{i}_{\lambda}}\!\left[
            \hspace{0.3mm}\mu_{\hspace{0.3mm}r_{1, \hspace{0.1mm}\lambda}},
            \hspace{0.3mm}\mu_{\hspace{0.3mm}r_{2, \hspace{0.1mm}\lambda}},
            \,\ldots,
            \,\mu_{\hspace{0.3mm}r_{\lambda, \hspace{0.1mm}\lambda}}
        \right]\!,
    \\\notag
        r_{0} = 1,
        \qquad 1 \le i_{\tau} \le N\!,
        \qquad 1 \le r_{\tau} \le R_{\tau},
        \qquad 1 \le \tau \le \lambda.
    \end{gather}
    Thus, the~equality~\eqref{ch3.eq19} with~$\lambda = D$ gives~us the~required TT-decomposition for~the~tensor with~the~elements of~the~form~\eqref{ch3.eq17}.

    \medskip
    The~base of~induction is trivial, if choose
    \begin{gather}
    \notag
        H^{(1)}_{r_{0}, \hspace{0.3mm}i_{1}, \hspace{0.3mm}r_{1}}
        := C^{(1)}_{i_{1}}\!\left[\hspace{0.3mm}\mu_{\hspace{0.3mm}r_{1}}\right]
        \equiv i_{1}^{\hspace{0.3mm}\mu_{\hspace{0.3mm}r_{1}}},
        \qquad
        1 \le i_{1} \le N\!,
        \quad 1 \le r_{1} \le R_{1}.
    \end{gather}
    Next we rewrite the~equation~\eqref{ch3.eq17} for~$D = 2, 3$:
    \begin{gather}
    \notag
        C^{(2)}_{i_{1}, \hspace{0.3mm}i_{2}}
        \!\left[\hspace{0.3mm}\mu_{1}, \hspace{0.3mm}\mu_{2}\right]
        \equiv \begin{bmatrix}
            i_{1}^{\hspace{0.3mm}\mu_{1}}
            &i_{1}^{\hspace{0.3mm}\mu_{2}}
        \end{bmatrix}
        \cdot
        \begin{bmatrix}
            i_{2}^{\hspace{0.3mm}\mu_{2}}
            \\i_{2}^{\hspace{0.3mm}\mu_{1}}
        \end{bmatrix}\!;
    \\\notag
        C^{(3)}_{i_{1}, \hspace{0.3mm}i_{2}, \hspace{0.3mm}i_{3}}
        \!\left[
            \hspace{0.3mm}\mu_{1}, \hspace{0.3mm}\mu_{2}, \hspace{0.3mm}\mu_{3}
        \right]
        \equiv \begin{bmatrix}
            i_{1}^{\hspace{0.3mm}\mu_{1}}
            &i_{1}^{\hspace{0.3mm}\mu_{2}}
            &i_{1}^{\hspace{0.3mm}\mu_{3}}
        \end{bmatrix}
        \cdot
        \begin{bmatrix}
            0
            & i_{2}^{\hspace{0.3mm}\mu_{3}}
            & i_{2}^{\hspace{0.3mm}\mu_{2}}
            \\i_{2}^{\hspace{0.3mm}\mu_{3}}
            &0
            & i_{2}^{\hspace{0.3mm}\mu_{1}}
            \\i_{2}^{\hspace{0.3mm}\mu_{2}}
            &i_{2}^{\hspace{0.3mm}\mu_{1}}
            &0
        \end{bmatrix}
        \cdot
        \begin{bmatrix}
            i_{3}^{\hspace{0.3mm}\mu_{1}}
            \\i_{3}^{\hspace{0.3mm}\mu_{2}}
            \\i_{3}^{\hspace{0.3mm}\mu_{3}}
        \end{bmatrix}\!.
    \end{gather}
    This representation allows us to~describe the~structure of~factors~$\mathscr{H}^{(\tau)}$
    for~each~$\tau$. If~the equality~\eqref{ch3.eq19} is already proven
    for~a~given~$\lambda$, then it is sufficient to~choose 
    \begin{gather}
    \notag
        H^{(\lambda + 1)}_{
            r_{\lambda}, \hspace{0.3mm}i_{\lambda + 1},
            \hspace{0.3mm}r_{\lambda + 1}
        }
        := \begin{cases}
            i_{\lambda + 1}^{
                \hspace{0.3mm}\mu_{\hspace{0.1mm}r_{\xi, \lambda + 1}}
            }\hspace{-1mm},
            &\hspace{4mm} \left\{
                r_{\xi, \hspace{0.1mm}\lambda + 1}
            \right\}
            \cup \mathcal{R}_{\lambda} = \mathcal{R}_{\lambda + 1},
        \\
            0, 
            &\hspace{4mm} \text{otherwise};
        \end{cases}
    \\\notag
        1 \le r_{\lambda} \le R_{\lambda},
        \qquad 1 \le i_{\lambda + 1} \le N\!,
        \qquad 1 \le r_{\lambda + 1} \le R_{\lambda + 1}.
    \end{gather}
    Now it is not hard to prove, that by~the~virtue of~the~proposed choice for~all~numbers~$
        1 \le i_{1}, \hspace{0.3mm}i_{2},
        \hspace{0.3mm}\ldots, \hspace{0.3mm}i_{\lambda + 1} \le N
    $ and~$1 \le r_{\lambda + 1} \le R_{\lambda + 1}$ we have the identity
    \begin{gather}
    \notag
        \sum_{r_{\lambda} = 1}^{R_{\lambda}}
            C^{(\lambda)}_{
                \mathbf{i}_{\lambda}
            }\!\left[
                \hspace{0.3mm}\mu_{\hspace{0.3mm}r_{1, \hspace{0.3mm}\lambda}},
                \hspace{0.3mm}\mu_{\hspace{0.3mm}r_{2, \hspace{0.3mm}\lambda}},
                \,\ldots,
                \,\mu_{\hspace{0.3mm}r_{
                    \lambda, \hspace{0.3mm}\lambda}
                }
            \right]
            \cdot
            H^{(\lambda + 1)}_{
                r_{\lambda}, \hspace{0.3mm}i_{\lambda + 1},
                \hspace{0.3mm}r_{\lambda + 1}
            }
        =
    \\\notag
        = 
        \sum_{\xi = 1}^{\lambda + 1}
            C^{(\lambda)}_{
                \mathbf{i}_{\lambda}
            }\!\left[
                \hspace{0.3mm}\mu_{
                    \hspace{0.3mm}r_{1, \hspace{0.3mm}\lambda + 1}
                },
                \,\ldots,
                \,\mu_{
                    \hspace{0.3mm}r_{\xi - 1, \hspace{0.3mm}\lambda + 1}
                },
                \hspace{0.3mm}\mu_{
                    \hspace{0.3mm}r_{\xi + 1, \hspace{0.3mm}\lambda + 1}
                },
                \,\ldots,
                \,\mu_{
                    \hspace{0.3mm}r_{\lambda + 1, \hspace{0.3mm}\lambda + 1}
                }
            \right]
            \cdot
            i_{\lambda + 1}^{\hspace{0.3mm}\mu_{
                \hspace{0.3mm}r_{\xi, \hspace{0.3mm}\lambda + 1}
            }}\hspace{-1mm}.
    \end{gather}
    Therefore, to verify an~induction step we just need to~use~the~identity~\eqref{ch3.eq18}, where it~is~necessary to~use~parameters~$
        \mu_{\hspace{0.3mm}r_{1, \hspace{0.1mm}\lambda + 1}},
        \hspace{0.1mm}\mu_{\hspace{0.3mm}r_{2, \hspace{0.1mm}\lambda + 1}},
        \,\ldots,
        \,\mu_{
            \hspace{0.3mm}r_{\lambda + 1, \hspace{0.1mm}\lambda + 1}
        }
    $ instead~of~parameters~$
        \mu_{1}, \hspace{0.3mm}\mu_{2},
        \hspace{0.3mm}\ldots, \hspace{0.3mm}\mu_{\lambda + 1}
    $  respectively. The last statement proves the~theorem.
\end{proof}


\section{Parallel algorithm and numerical experiments}
\qquad In our work we exploit organization of parallel computations along particle size coordinate with dimension $N$. It is worth to note that alternative way of parallelization of our approach along TT-ranks leads to dramatic overheads in terms of data exchanges and collective operations and does not lead to speedup of computations.

\begin{figure}[h!]
    \centering
    \includegraphics[scale=0.25]{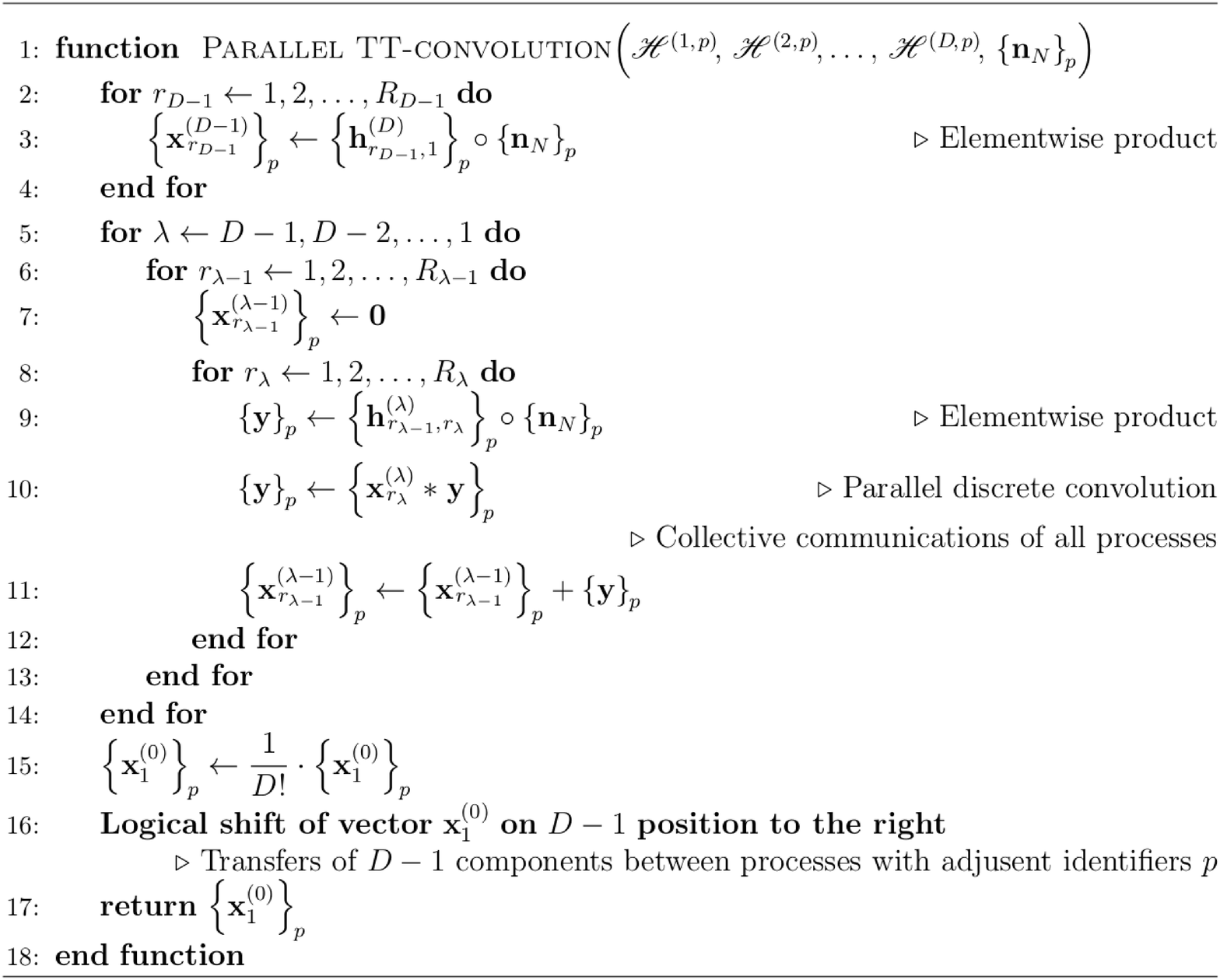}
    \caption{Parallel algorithm for $\mathcal{P}^{(D\hspace{0.3mm})}$.}
    \label{algo_convol}
\end{figure}

\begin{figure}[h!]
    \centering
    \includegraphics[scale=0.25]{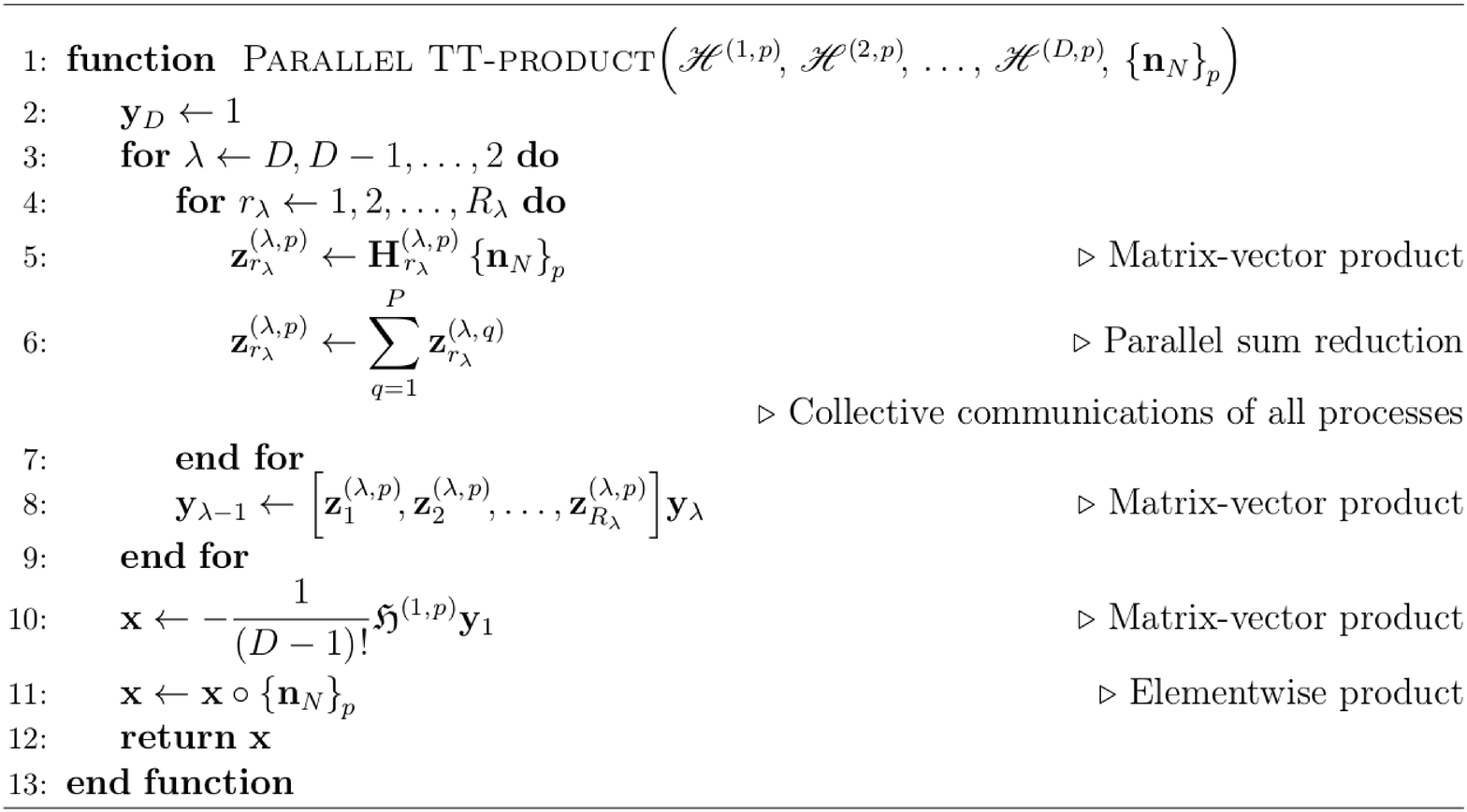}
    \caption{Parallel algorithm for $\mathcal{Q}^{(D\hspace{0.3mm})}$. }
    \label{algo_product}
\end{figure}

Let us assume that we have $P$ processors and number $N$ of studied kinetic equations is divisible by $P$. Thus, we introduce the following notation for getting $p$-th block of taken vector
~\hbox{$
    \mathbf{a}_{N}
    = \left[s
        a_{1},
        \hspace{0.3mm}a_{2},
        \hspace{0.3mm}\ldots,
        \hspace{0.3mm}a_{N}
    \right]^{T}
$}:
\begin{gather}
\notag
    \left\{\mathbf{a}_{N}\right\}_{\hspace{-0.3mm}p}
    := \left[
        a_{\left(p - 1\right)N/P + 1},
        \,a_{\left(p - 1\right)N/P + 2},
        \,\ldots,
        \,a_{\hspace{0.3mm}pN/P}
    \right]^{T}\hspace{-1mm},
    \qquad
    1 \leqslant p \leqslant P\!.
\end{gather}
With use of those notations we can denote blocks
и~$\mathscr{H}^{(\lambda, \hspace{0.3mm}p)}$
of cores~$\mathscr{H}^{(\lambda)}$ of TT-decomposition of kernel kinetic coefficients ~$\mathscr{C}^{(D)}$\! which will be used at processor with number ~$p$ (~\hbox{$1 \le p \le P$})~:
\begin{gather}
\notag
    \mathbf{h}^{(\lambda)}_{
        r_{\lambda - 1}, \hspace{0.3mm}r_{\lambda}
    }
    := \left[
        H^{(\lambda)}_{
            r_{\lambda - 1}, \hspace{0.3mm}1,
            \hspace{0.3mm}r_{\lambda}
        },
        \hspace{0.3mm}H^{(\lambda)}_{
            r_{\lambda - 1}, \hspace{0.3mm}2,
            \hspace{0.3mm}r_{\lambda}
        },
        \hspace{0.3mm}\ldots,
        \hspace{0.3mm}H^{(\lambda)}_{
            r_{\lambda - 1}, \hspace{0.3mm}N\!,
            \hspace{0.3mm}r_{\lambda}
        }
    \right]^{T}
    \in \mathbb{R}^{N}\!,
\\\notag
    \mathbf{H}^{(\lambda, \hspace{0.3mm}p)}_{r_{\lambda}}
    := \left[
        \left\{
            \mathbf{h}^{(\lambda)}_{1, \hspace{0.3mm}r_{\lambda}}
        \right\}_{\!p}\!,
        \hspace{0.3mm}\left\{
            \mathbf{h}^{(\lambda)}_{2, \hspace{0.3mm}r_{\lambda}}
        \right\}_{\!p}\!,
        \hspace{0.3mm}\ldots,
        \hspace{0.3mm}\left\{
            \mathbf{h}^{(\lambda)}_{R_{\lambda - 1}, \hspace{0.3mm}r_{\lambda}}
        \right\}_{\!p}\hspace{-0.3mm}
    \right]^{T}
    \in \mathbb{R}^{R_{\lambda - 1} \times N/P}\!,
\\\notag
    \mathscr{H}^{(\lambda, \hspace{0.3mm}p)}
    := \left[
        \left\{\mathbf{h}^{(\lambda)}_{r_{\lambda - 1},
        \hspace{0.3mm}r_{\lambda}}\right\}_{\!p}\hspace{-0.3mm}
    \right] 
    \equiv \Bigl[
        \mathbf{H}^{(\lambda, \hspace{0.3mm}p)}_{
            \hspace{0.3mm}r_{\lambda}
        }
    \Bigr]
    \in \mathbb{R}^{
        R_{\lambda - 1} \times N/P \times R_{\lambda}
    }\hspace{-0.3mm}.
\end{gather}

The algorithm for operator $\mathcal{P}^{(d\hspace{0.3mm})}$ is presented in Fig. \ref{algo_convol} and for  $\mathcal{Q}^{(d\hspace{0.3mm})}$ in Fig. \ref{algo_product}. On the input algorithms require to have blocks $\mathscr{H}^{(1,i)}$ of TT-decomposition for kinetic coefficients and vector of concentrations.  For time-integration of the Cauchy problem we utilize standard explicit second order Runge-Kutta method, hence, each time-step requires two evaluations of $\mathcal{P}^{(d\hspace{0.3mm})}$ and $\mathcal{Q}^{(d\hspace{0.3mm})}$. 

We present results of benchmarks of presented algorithm for the generalized Brownian coefficients in Table \ref{table_scalab}. In our experiments we used ClusterFFT library included into Intel~MKL\texttrademark. As soon as FFT is a dominating operation in our algorithm in terms of complexity, we obtain similar performance of our code to performance of ClusterFFT library. From these experiments we obtain acceleration of calculations by order of magnitude. This allows us to consider a broader class of problems of potential interest which can be studied in modest computational time.

\begin{table}
\caption{Speedup of computations with use of ClusterFFT operation for pure ternary aggregation in case of 3-dimensional generalized Brownian kernel with $N = 2^{19}$ equations. Benchmark for 100 time-integration steps with use of second order Runge-Kutta method.  ``Zhores'' supercomputer of Skolkovo Institute of Science and Technology
tables.}\label{table_scalab}
\begin{center}
\begin{tabular}{|c|c|c|}
\hline
~Number of CPU-cores~ &  ~time, sec ~ &~ Speedup~~\\
\hline
1    & 257.80   & 1.00    \\  
2    & 147.62  & 1.75 \\
4    & 80.21    & 3.21 \\
8    & 43.65    & 5.91 \\
16   & 22.63    & 11.39 \\
32   & 14.83    & 17.38 \\
64   & 13.15    & 19.60 \\
128  & 12.22    & 21.09  \\
\hline
\end{tabular}
\end{center}
\end{table}

\section{Conclusions}

\qquad In this paper we present recent developments of tensor based methods for modelling of multi-particle aggregation. We prove estimates of TT-ranks for generalized Brownian kinetic coefficients depending only on dimensionality $D$ but not mode-sizes $N$ of used arrays. We also propose an efficient way of parallel implementation of TT-based approach and demonstrate preliminary results of its parallel scalability.

In our work we used ``Zhores'' supercomputer installed at Skolkovo Institute of Science and Technology \cite{Zhores}. We also would like to acknowledge Talgat Daulbaev for an idea of representation of generalized Brownian coefficients in TT-format in case of $D=3$. The work was supported by the Russian Science Foundation, grant 19-11-00338.

%
%
%
%

\end{document}